\documentclass[11pt]{amsart}
\usepackage{amssymb}
\usepackage{amsmath}
\usepackage{graphicx}
\usepackage[latin1]{inputenc}
\usepackage{amsfonts}
\usepackage[spanish,english]{babel}
\usepackage{comma}
\usepackage[active]{srcltx}
\providecommand{\U}[1]{\protect\rule{.1in}{.1in}}
\newtheorem{theorem}{Theorem}
\theoremstyle{plain}

\newtheorem{corollary}[theorem]{Corollary}

\numberwithin{equation}{section}
\begin{document}
\title[Dirichlet's theorem and Midy's property]{On a particular case of the Dirichlet's theorem and the Midy's property}
\author[J.H. Castillo]{John H. Castillo}
\address{John H. Castillo, Departamento de Matemáticas y Estadística, Universidad de Nariño, San Juan de Pasto-Colombia}
\email{jhcastillo@gmail.com}
\author[G. García-Pulgarín]{Gilberto Garc\'\i a-Pulgar\'in}
\address{Gilberto Garc\'\i a-Pulgar\'in, Universidad de Antioquia, Medellín-Colombia}
\email{gigarcia@ciencias.udea.edu.co}
\author[J.M Velásquez Soto]{Juan Miguel Vel\'asquez-Soto}
\address{Juan Miguel Vel\'asquez-Soto, Departamento de Matemáticas, Universidad del Valle, Cali-Colombia}
\email{jumiveso@univalle.edu.co} 
\date{}

\begin{abstract}
We give new characterizations of the Midy's property and using these
results we obtain a new proof of a special case of the Dirichlet's
theorem about primes in arithmetic progression.

\end{abstract}
\maketitle

\section{Preliminaries}

Let $b$ be a positive integer greater than $1$, $b$ will denote the
base of numeration, $N$ a positive integer relatively prime to $b$,
i.e $(N,b)=1$, $\left\vert b\right\vert _{N}$ the order of $b$ in
the multiplicative group $\mathbb{U}_{N}$ of positive integers less
than $N$ and relatively primes to $N,$ and $x\in\mathbb{U}_{N}$. It
is well known that when we write the fraction $\frac{x}{N}$ in base
$b$, it is periodic. By period we mean the smallest repeating
sequence of digits in base $b$ in such expansion, it is easy to see
that $\left\vert b\right\vert _{N}$ is the length of the
period of the fractions $\frac{x}{N}$ (see Exercise 2.5.9 in \cite{Nathanson}%
). Let $d,\,k$ be positive integers with $\left\vert b\right\vert
_{N}=dk$, $d>1$ and $\frac{x}{N}=0.\overline{a_{1}a_{2}\cdots
a_{\left\vert b\right\vert _{N}}}$ where the bar indicate the period
and $a_{i}$'s are digits in base $b$. We separate the period
${a_{1}a_{2}\cdots a_{\left\vert b\right\vert _{N}}}$ in $d$ blocks
of length $k$ and let
\[
A_{j}=[a_{(j-1)k+1}a_{(j-1)k+2}\cdots a_{jk}]_{b}%
\]
be the number represented in base $b$ by the $j$-th block and $S_{d}%
(x)=\sum\limits_{j=1}^{d}A_{j}$. If for all $x\in\mathbb{U}_{N}$,
the sum $S_{d}(x)$ is a multiple of $b^{k}-1$ we say that $N$ has
the Midy's property for $b$ and $d$. It is named after E. Midy
(1836), to read historical aspects about this property see
\cite{Lewittes} and its references.

We denote with $\mathcal{M}_{b}(N)$ the set of positive integers $d$
such that $N$ has the Midy's property for $b$ and $d$ and we will
call it the Midy's set of $N$ to base $b$. As usual, let
$\nu_{p}(N)$ be the greatest exponent of $p$ in the prime
factorization of $N$.

For example $13$ has the Midy's property to the base $10$ and $d=3$,
because $|13|_{10}=6$, $1/13=0.\overline{076923}$ and $07+69+23=99$.
Also, $75$ has the Midy's property to the base $8$ and $d=4$, since
$|75|_{8}=20$, $1/75=[0.\overline{00664720155164033235}]_8$ and
$[00664]_8+[72015]_8+[51640]_8+[33235]_8=2*(8^5-1)$. But $75$ does
not have the Midy's property to $8$ and $5$. Actually, we can see
that $\mathcal{M}_{10}(13)=\{2,3,6\}$ and $\mathcal{M}_{8}(75)=\{4,
20\}$.

In \cite{garcia09} is given the following characterization of Midy's
property.

\begin{theorem}
\label{ppl2} If $N$ is a positive integer and $\left\vert
b\right\vert _{N}=kd$, then $d\in\mathcal{M}_{b}(N)$ if and only if
$\nu_{p}(N)\leq\nu _{p}(d)$ for all prime divisor $p$ of $(b^{k}-1,\
N)$.
\end{theorem}

The next theorem is a different way to write Theorem \ref{ppl2}.

\begin{theorem}
\label{ppl3}Let $N$ be a positive integer and $d$ a divisor of
$\left\vert b\right\vert _{N}$. \ The following statements are
equivalent

\begin{enumerate}
\item $d\in\mathcal{M}_{b}(N)$

\item For each prime divisor $p$ of $N$ such that $\nu_{p}\left(  N\right)
>\nu_{p}\left(  d\right)  $, there exists a prime $q$ divisor of $\left\vert
b\right\vert _{N}$ that satisfies $\nu_{q}\left(  \left\vert
b\right\vert _{p}\right)  >\nu_{q}\left(  \left\vert b\right\vert
_{N}\right)  -\nu _{q}\left(  d\right)  $.
\end{enumerate}
\end{theorem}

In \cite{trio} the authors prove the following theorem.

\begin{theorem}
\label{coro}Let $d_{1}$, $d_{2}$ be divisors of $\left\vert
b\right\vert _{N}$ and assume that $d_{1}\mid d_{2}$ and
$d_{1}\in\mathcal{M}_{b}(N)$, then $d_{2}\in\mathcal{M}_{b}(N)$.
\end{theorem}

The following result has a big influence on our work, it is Theorem
3.6 in \cite{Nathanson}.

\begin{theorem}
\label{poten}Let $p$ be an odd prime not divisor of $b$, $m=\nu_{p}%
(b^{\left\vert b\right\vert _{p}}-1)$ and let $t$ a positive
integer, then
\[
\left\vert b\right\vert _{p^{t}}=%
\begin{cases}
\left\vert b\right\vert _{p} & \text{ if }t\leq m,\\
& \\
p^{t-m}\left\vert b\right\vert _{p} & \text{ if \ }t>m.
\end{cases}
\]

\end{theorem}

The Dirichlet's theorem about primes in arithmetic progression
states that there are infinitely many primes in any arithmetic
progression of initial term  $a$ and common difference $l$, with
$(a,l)=1$. It was conjectured by Carl Friedrich Gauss and it was
first proved in 1826 by Peter Gustav Lejeune Dirichlet. The theorem
which establishes that there are infinitely many primes is a
particular case of the Dirichlet's theorem, taking $a=1$ and $l=1$.

A related result, to the Dirichlet's theorem, proved by B. Green and
T. Tao in 2004, guarantees that we can find lists of primes in
arithmetic progression of arbitrary length; see \cite{GreenTao}.

The biggest list  of primes in arithmetic progression has 26 terms,
it was founded in April 12, 2010 by Benot Perichon and PrimeGrid. It
is formed by the numbers $43142746595714191 + 23681770\cdot
23\#\cdot n$, for $0 \leq n \leq 25$ where $a\#$  denotes the
primorial of $a$, it is the product of all primes less or equal to
$a$; see \cite{PAR}.
\section{Other characterizations of Midy's Property.}

In this section, we will study some consequences of Theorem
\ref{ppl3}.

\begin{theorem}
\label{guel}Let $N$ be a positive integer and $\ \left\vert
b\right\vert _{N}=kd$. \ If, for all prime divisor $p$ of $N$, we
have $\nu_{p}\left( N\right)  >\nu_{p}\left(  d\right)  $, then the
followings statements are equivalent

\begin{enumerate}
\item $\left(  b^{k}-1,\ N\right)  =1$

\item $d\in\mathcal{M}_{b}(N)$

\item For each prime divisor $p$ of $N$, there exists a prime $q$ divisor of
$d$ such that $\nu_{q}\left(  |b|_{p}\right)  >\nu_{q}\left(
|b|_{N}\right) -\nu_{q}\left(  d\right)  \text{.}$
\end{enumerate}
\end{theorem}

\begin{proof}
The equivalence between (2) and (3) is immediate from Theorem
\ref{ppl3}.

By Theorem \ref{ppl2} we get that (1) implies (2). Now we prove that
(2) implies (1). \ Suppose that $d\in\mathcal{M}_{b}(N)$ and let
$g=\left( b^{k}-1,\ N\right)  $. If there exists a prime divisor $p$
of $g$, from Theorem \ref{ppl2}, we have $0<\nu_{p}\left(  N\right)
\leq\nu_{p}\left( d\right)  $ and this is impossible because
$\nu_{p}\left(  d\right)  <\nu _{p}\left(  N\right)  $. Therefore
$\left(  b^{k}-1,\ N\right)  =1$.
\end{proof}

We now will study Theorem 4 of \cite{Lewittes}, which is attributed
by its author to M. Jenkins (1867). Let $p_{1},p_{2}\ldots,p_{t}$ be
different primes
such that $d\in\mathcal{M}_{b}(p_{i})$ \ for each $i$ and let $h_{1}%
,h_{2},\ldots,h_{t}$ be positive integers, when does $N=p_{1}^{h_{1}}%
p_{2}^{h_{2}}\cdots p_{t}^{h_{t}}$ have the Midy's property for $b$
and $d$?. The Jenkins' Theorem gives the answer to this question and
the same is independent from the $h_{i}$'s.

By simplicity we go to study the case when $t=3$, but the argument
is true for any $t$. Let $p_{1},p_{2}, p_{3}$ be different primes
such that $d\in
\mathcal{M}_{b}(p_{i})$, $\left\vert b\right\vert _{p_{i}}=dk_{i}$, $m_{i}%
=\nu_{p_{i}}\left(  b^{\left\vert b\right\vert _{p_{i}}}-1\right)  $
for $i=1,2,3$ then
\[
\left\vert b\right\vert _{N}=d\left[  p_{1}^{h_{1}-m_{1}}k_{1}\text{, }%
p_{2}^{h_{2}-m_{2}}k_{2}\text{, }p_{3}^{h_{3}-m_{3}}k_{3}\right]  =dk\text{.}%
\]
We have to check up the prime divisors of $\left(  b^{k}-1\text{,
}N\right) $, to determine when $d\in\mathcal{M}_{b}(N)$. As $d$ is a
divisor of $\left\vert b\right\vert _{p_{1}}$ thus $d\leq p_{1}-1$
and if, say, $p_{1}\mid$ $\left(  b^{k}-1\text{, }N\right)  $ we get
that $h_{1}=\nu _{p_{1}}\left(  N\right)  >0=\nu_{p_{1}}\left(
d\right)  $ and so $N$ does not have the Midy's property for $b$ and
$d$. In consequence $d\in \mathcal{M}_{b}(N)$ if and only if $\left(
b^{k}-1\text{, }N\right)  =1$. It is clear that $\left(
b^{k}-1\text{, }N\right)  =1$ is equivalent to say that for each
$i$, $\left\vert b\right\vert _{p_{i}}\nmid k$. We will see when
this fact is verified. Let $d=\prod\limits_{i=1}^{s}q_{i}^{r_{i}}$
be the prime decomposition of $d$, for each $i=1,2,3$ take
$c_{i}=\nu_{d}\left( k_{i}\right)  $ and
\begin{align*}
p^{h_{1}-m_{1}}k_{1}  &  =\left(  d^{c_{1}}\prod\limits_{i=1}^{s}q_{i}%
^{\alpha_{i}^{\left(  1\right)  }}\right)  y_{1}\\
p^{h_{2}-m_{2}}k_{2}  &  =\left(  d^{c_{2}}\prod\limits_{i=1}^{s}q_{i}%
^{\alpha_{i}^{\left(  2\right)  }}\right)  y_{2}\\
p^{h_{3}-m_{3}}k_{3}  &  =\left(  d^{c_{3}}\prod\limits_{i=1}^{s}q_{i}%
^{\alpha_{i}^{\left(  3\right)  }}\right)  y_{3}%
\end{align*}
where $(q_{i},y_{j})=1$ for $i=1,2,\ldots, s$ and $j=1,2,3$. Then
\begin{align*}
k  &  =\left[  p^{h_{1}-m_{1}}k_{1}\text{,
}p^{h_{2}-m_{2}}k_{2}\text{,
}p^{h_{3}-m_{3}}k_{3}\right] \\
&  =\left[  d^{c_{1}}\prod\limits_{i=1}^{s}q_{i}^{\alpha_{i}^{\left(
1\right)  }}\text{,
}d^{c_{2}}\prod\limits_{i=1}^{s}q_{i}^{\alpha_{i}^{\left( 2\right)
}}\text{, }d^{c_{3}}\prod\limits_{i=1}^{s}q_{i}^{\alpha_{i}^{\left(
3\right)  }}\right]  \left[  y_{1}\text{, }y_{2}\text{,
}y_{3}\right]
\text{.}%
\end{align*}
We want, as was said before, that $k$ not be divisible for any
$\left\vert b\right\vert _{p_{i}}$, $i=1,2,3$. Now if, say,
$k=\left\vert b\right\vert _{p_{1}}l$ with $l$ integer, then
$k=\left\vert b\right\vert _{p_{1}}l=\left( k_{1}d\right)  l$
therefore $\frac{k}{k_{1}}$ is a multiple of $d$, and as the
$y_{i}$'s are relatively primes with $d$ this is equivalent to
\[
\frac{\left[
d^{c_{1}}\prod\limits_{i=1}^{s}q_{i}^{\alpha_{i}^{\left( 1\right)
}}\text{, }d^{c_{2}}\prod\limits_{i=1}^{s}q_{i}^{\alpha_{i}^{\left(
2\right)  }}\text{,
}d^{c_{3}}\prod\limits_{i=1}^{s}q_{i}^{\alpha_{i}^{\left( 3\right)
}}\right]  }{d^{c_{1}}\prod\limits_{i=1}^{s}q_{i}^{\alpha
_{i}^{\left(  1\right)  }}}\equiv0\ \operatorname{mod}\ d
\]

From the above analysis follows the next theorem.

\begin{theorem}
[Jenkins' Theorem]Let $p_{1},p_{2},\ldots, p_{t}$ be different
primes such that $d\in\mathcal{M}_{b}(p_{i})$ for each $i$, let
$h_{1},h_{2}, \ldots, h_{t}$ be positive integers and
$N=p_{1}^{h_{1}}p_{2}^{h_{2}}\cdots p_{t}^{h_{t}}$. With the
notations introduced above we obtain that $d\in\mathcal{M}_{b}(N)$
if and only if, for each $j=1,2,\ldots,t$ is satisfied
\[
\frac{\left[
d^{c_{1}}\prod\limits_{i=1}^{s}q_{i}^{\alpha_{i}^{\left( 1\right)
}}\text{, }d^{c_{2}}\prod\limits_{i=1}^{s}q_{i}^{\alpha_{i}^{\left(
2\right)  }}\text{, \ldots,
}d^{c_{t}}\prod\limits_{i=1}^{s}q_{i}^{\alpha
_{i}^{\left(  t\right)  }}\right]  }{d^{c_{t}}\prod\limits_{i=1}^{s}%
q_{i}^{\alpha_{i}^{\left(  j\right)  }}}\not \equiv 0\
\operatorname{mod}\ d
\]

\end{theorem}

Our next result has a similar flavor of the Jenkins' Theorem,
although its statement and proof are simpler.

\begin{theorem}
Let $N,q,v$ be integers with $q$ prime and $v>0$. Then $q^{v}\in
\mathcal{M}_{b}(N)$ if and only if
$N=q^{n}p_{1}^{h_{1}}p_{2}^{h_{2}}\cdots p_{l}^{h_{l}}$ where $n$ is
a non-negative integer, $p_{i}$'s are different primes and $h_{i}$'s
are non-negatives integers not all zero, verifying $0\leq n\leq v$,
$\nu_{q}(\left\vert b\right\vert _{p_{i}}) >0$ and
\[
\max\limits_{1\leq i \leq l}\left\{  n-m\text{, }\nu_{q}( \left\vert
b\right\vert _{p_{i}}) \right\}  -v<\min\limits_{1\leq i \leq
l}\left\{
\nu_{q}( \left\vert b\right\vert _{p_{i}}) \right\}  \text{ }%
\]
where $m=\nu_{q}(b^{\left\vert b\right\vert _{q}}-1)$.
\end{theorem}

\begin{proof}
We write $\left\vert b\right\vert _{N}=q^{t}k$, with $(q,k)=1$ and
$t\geq v$.
Let us denote $g=\left(  b^{kq^{t-v}}-1, N\right)  $. Suppose that $q^{v}%
\in\mathcal{M}_{b}(N)$. By Theorem \ref{ppl2} we know that $g$ can
not be divisible by other prime different from $q$ and that
$\nu_{q}\left(  N\right) \leq v$. Let $p\neq q$ be a prime divisor
of $N$. Because $p$ not divides $g$, we have $\left\vert
b\right\vert _{p}\nmid kq^{t-v}$ and thus $\nu_{q}( \left\vert
b\right\vert _{p}) >t-v\geq0$ and it is easy to see that
$t=\max\limits_{1\leq i \leq l}\left\{  n-m\text{, }\nu_{q}(
\left\vert b\right\vert _{p_{i}}) \right\}  $. Therefore for all
prime divisor $p$ of $N$ we get that $\nu_{q}( \left\vert
b\right\vert _{p}) >\max\limits_{1\leq i \leq l}\left\{  n-m\text{,
}\nu_{q}( \left\vert b\right\vert _{p_{i}}) \right\} -v$.

Conversely, from the hypothesis the unique prime divisor of $N$ that
could be
a divisor of $g$ is $q$ and thus Theorem \ref{ppl2} implies that $q^{v}%
\in\mathcal{M}_{b}(N)$ .
\end{proof}

As a special case, if in the above theorem we do $v=1$ and since for
any $p$ prime divisor of $N$ we have $\nu_{q}\left(  \left\vert
b\right\vert _{N}\right)  \geq\nu_{q}\left(  |b|_{p}\right)  $ and
thus we obtain the next corollary.

\begin{corollary}
\label{co1}Let $N$ be a positive integer and let $q$ be a prime
divisor of $\left\vert b\right\vert _{N}$, then
$q\in\mathcal{M}_{b}(N)\ $ if and only if

\begin{enumerate}
\item If $\left(  N,\ q\right)  =1$, then $\nu_{q}(\left\vert b\right\vert
_{p})=\nu_{q}\left(  \left\vert b\right\vert _{N}\right)  $ for all
$p$ prime divisor of $N$.

\item If $\left(  N,\ q\right)  >1$, then $q^{2}$ not divides $N$ and
$\ \nu_{q}(\left\vert b\right\vert _{p})=\nu_{q}\left(  \left\vert
b\right\vert _{N}\right)  $ for all $p$ prime divisor of $N$
different from $q$.
\end{enumerate}
\end{corollary}

Note that, from the last theorem, the smallest number $N$ such that $q^{v}%
\in\mathcal{M}_{b}(N)$ has to be a prime $P$ which satisfies $\nu
_{q}(\left\vert b\right\vert _{P}) >0$ and $q^{v}|\ |b|_{P}$. We
obtain the below corollary, recalling that $|b|_{P}$ divides $P-1$.

\begin{corollary}
If $q$ is a prime and $v$ is a positive integer, then the smallest
integer $N$ such that $q^{v}\in\mathcal{M}_{b}(N)$ is a prime
congruent with $1$ $\operatorname{mod}$ $q^{v}$.
\end{corollary}

As a consequence of the above corollary we have a particular case of
the Dirichlet's Theorem about primes in arithmetic progressions.

\begin{corollary}
If $q$ is a prime and $v$ is a positive integer, there are
infinitely many primes which are congruent to $1$ modulo $q^{v}$.
\end{corollary}

\begin{proof}
By the last corollary there exists a prime $P_{1}$ congruent with
$1$ modulo $q^{v}$ and which satisfies
$q^{v}\in\mathcal{M}_{b}(P_{1})$. Take $t_{1}$ an integer such that
$q^{t_{1}v}>P_{1}$. Once again, we can find a prime $P_{2}$
congruent with $1$ modulo $q^{t_{1}v}$ such that $q^{t_{1}v}\in\mathcal{M}%
_{b}(P_{2})$ and so on.
\end{proof}

\section*{Acknowledgements}
The authors are members of the research group: \'Algebra, Teor\'ia
de N\'umeros y Aplicaciones, ERM. J.H. Castillo was partially
supported by CAPES, CNPq from Brazil and Universidad de Nariño from
Colombia. J.M. Velásquez-Soto was partially supported by CONICET
from Argentina and Universidad del Valle from Colombia.

\bibliographystyle{amsalpha}
\bibliography{bibliografiaggp}

\providecommand{\bysame}{\leavevmode\hbox to3em{\hrulefill}\thinspace}
\providecommand{\MR}{\relax\ifhmode\unskip\space\fi MR }
\providecommand{\MRhref}[2]{%
  \href{http://www.ams.org/mathscinet-getitem?mr=#1}{#2}
}
\providecommand{\href}[2]{#2}
\begin{thebibliography}{CGPVS11}

\bibitem[And12]{PAR}
Jens~Kruse Andersen, \emph{Primes in arithmetic progression records}, retrivied
  05 March 2012, URL
  \url{http://users.cybercity.dk/~dsl522332/math/aprecords.htm} (2012).

\bibitem[CGPVS11]{trio}
John~H. Castillo, Gilberto Garc{\'{\i}}a-Pulgar{\'{\i}}n, and {Juan}~Miguel
  Vel{\'a}squez-Soto, \emph{{Structure of associated sets to Midy's Property}},
  accepted to publication in Matemáticas: Enseñanza Universitaria,
  arXiv:1110.3308v2 [math.NT] (2011).

\bibitem[GPG09]{garcia09}
Gilberto Garc{\'{\i}}a-Pulgar{\'{\i}}n and Hern{\'a}n Giraldo,
  \emph{Characterizations of {M}idy's property}, Integers \textbf{9} (2009),
  A18, 191--197. \MR{MR2506150}

\bibitem[GT04]{GreenTao}
Ben Green and Terence Tao, \emph{{The primes contain arbitrarily long
  arithmetic progressions}}, arXiv:math/0404188v6 [math.NT] (2004).

\bibitem[Lew07]{Lewittes}
Joseph Lewittes, \emph{Midy's theorem for periodic decimals}, Integers
  \textbf{7} (2007), A2, 11 pp. (electronic). \MR{MR2282184 (2008c:11004)}

\bibitem[Nat00]{Nathanson}
Melvyn~B. Nathanson, \emph{Elementary methods in number theory}, Graduate Texts
  in Mathematics, vol. 195, Springer-Verlag, New York, 2000. \MR{1732941
  (2001j:11001)}

\end{thebibliography}

\end{document}